\documentclass[a4paper,twoside,10pt]{amsart}
\usepackage{amsmath,amssymb,amsthm,amsfonts}
\usepackage[dvips]{graphicx}
\usepackage[all]{xy}
\usepackage{mathrsfs}

\usepackage[colorlinks, 	
 pdfpagelabels,
 pdfstartview = FitH,
 bookmarksopen = true,
 bookmarksnumbered = true,
 linkcolor = blue,
 plainpages = false,
 hypertexnames = false,
 citecolor = blue]{hyperref}

\usepackage{natbib}

\bibpunct{[}{]}{,}{a}{}{;}

\newtheoremstyle{dtheorem}{3 mm}{1 mm}{\itshape}{}{\bfseries}{.}{ }
  {\thmnumber{(#2) }\thmname{#1}\thmnote{ \mdseries(#3)\bfseries}}
\newtheoremstyle{ddef}{3 mm}{1 mm}{\normalfont}{}{\bfseries}{.}{ }
  {\thmnumber{(#2) }\thmname{#1}\thmnote{ \mdseries(#3)\bfseries}}
\newtheoremstyle{dremark}{3 mm}{1 mm}{\normalfont}{}{\itshape}{}{ }
  {\thmnumber{\upshape\bfseries(#2) }\itshape\mdseries\thmname{#1}.\thmnote{\;\mdseries #3 ---}}

\swapnumbers
\theoremstyle{dtheorem}
\newtheorem{thm}{Theorem}[section]
\newtheorem*{nnthm}{Theorem}
\newtheorem{prop}[thm]{Proposition}

\newtheorem{lemma}[thm]{Lemma}
\newtheorem{cor}[thm]{Corollary}

\theoremstyle{definition}
\newtheorem{defn}{Definition}[section]

\theoremstyle{remark}
\newtheorem{rem}[thm]{Remark}

\newtheorem*{conj}{Conjecture}

\numberwithin{equation}{section}

\newcommand{\thmref}[1]{Theorem~\ref{#1}}
\newcommand{\defref}[1]{Definition~\ref{#1}}

\newcommand{\proref}[1]{Proposition~\ref{#1}}
\newcommand{\lemref}[1]{Lemma~\ref{#1}}
\newcommand{\corref}[1]{Corollary~\ref{#1}}

\newcommand{\sheaf}[1]{\mathscr{#1}}
\newcommand{\cat}[1]{\mathcal{#1}}
\newcommand{\mdl}[1]{\mathcal{#1}}
\newcommand{\ideal}[1]{\mathfrak{#1}}

\DeclareMathOperator{\Pic}{Pic}

\DeclareMathOperator{\Alb}{Alb}
\DeclareMathOperator{\Aut}{Aut}

\DeclareMathOperator{\Spec}{Spec}
\DeclareMathOperator{\Spf}{Spf}
\DeclareMathOperator{\Id}{Id}

\DeclareMathOperator{\pr}{pr}

\DeclareMathOperator{\Lie}{Lie}
\DeclareMathOperator{\Def}{\cat D\mathit{ef}}
\DeclareMathOperator{\Fib}{\cat F\mathit{ib}}
\DeclareMathOperator{\Jac}{\cat J\mathit{ac}}
\DeclareMathOperator{\GL}{GL}
\DeclareMathOperator{\Sym}{Sym}
\DeclareMathOperator{\Proj}{Proj}

\renewcommand{\Im}{\text{Im}}

\renewcommand{\O}{\sheaf{O}}
\DeclareMathOperator{\Hom}{Hom}

\DeclareMathOperator{\ord}{ord}

\DeclareMathOperator{\tr}{tr}

\DeclareMathOperator{\Alg}{\cat Alg}

\newcommand{\Z}{\mathbb{Z}}

\newcommand{\Q}{\mathbb{Q}}
\newcommand{\F}{\mathbb{F}}
\newcommand{\G}{\mathbb{G}}

\renewcommand{\P}{\mathbb{P}}

\newcommand{\MGamma}[1]{\mdl{M}_{\Gamma({#1})}}
\newcommand{\Zm}[1]{\Z/{#1}\Z}

\title[Deformations of elliptic fibre bundles]{Deformations of elliptic fibre bundles in positive characteristic}
\author[H. Partsch]{Holger Partsch}
\address[H. Partsch]{Mathematisches Institut\\
Heinrich-Heine-Universit\"at\\
D-40225 D\"usseldorf\\
Deutschland\\
Fax: ++49(0)211-81-15233}
\email{partsch@math.uni-duesseldorf.de}
\thanks{The author was supported by the Sonderforschungsbereich/Transregio 45
``Periods, moduli spaces and arithmetic of algebraic varieties''}


\dedicatory{\today}

\begin{document}

\maketitle

\begin{abstract}
 We study the deformation theory of elliptic fibre bundles over curves in positive characteristics.
 As applications, we give examples of non-liftable elliptic surfaces in characteristic two and three, which answers a question of Katsura and Ueno. Also, we construct a class of elliptic fibrations, whose liftability is equivalent to a conjecture of F. Oort concerning the liftability of automorphisms of curves.  Finally, we classify deformations of bielliptic surfaces.
\end{abstract}

\bibliographystyle{pgalpha}

%
%
%
%
%
%
%
%
%
%

\section{Introduction}
In their seminal paper on elliptic surfaces in characteristic $p$ \cite{KU}, Katsura and Ueno asked if every elliptic surface of Kodaira dimension one, over a field of positive characteristic, can be lifted to characteristic zero.
This question is of interest, because the known examples of non-liftable surfaces are either quasi-elliptic or of general type.
On the other hand, every surface of Kodaira dimension less than or equal to zero is know to be liftable.
In this article, we give the first examples of elliptic Kodaira dimension one surfaces, which are non-liftable.

%
%
%

Recall that every surface of Kodaira dimension one has a unique elliptic or quasi-elliptic fibration. A fibration is called quasi-elliptic if the generic fibre is a cuspidal curve of arithmetic genus one.
%

Our approach to construct non-liftable elliptic surfaces works as follows:
We classify all possible deformations of a given surface, and a posteriori conclude that there are only deformations over Artinian rings of characteristic $p$. In particular, this excludes liftability to characterstic zero.

To make this work, we have to choose a class of surfaces, with a sufficiently easy deformation theory, and being on the other hand rich enough to provide the examples we are looking for. As it turns out, the right objects are elliptic fibre bundles over curves.

Let us fix a field $k$ and a curve $C$ over $k$.
By an \emph{elliptic fibre bundle} over $C$, we understand an elliptic fibration $X \to C$ which is locally trivial in the \'etale topology (Definition \ref{defn-efbundle}). 

An elliptic fibre bundle is called \emph{Jacobian} if it has a section. Examples of Jacobian elliptic bundles can be constructed as follows:
Let $E$ be an elliptic curve over $k$, and let $\Gamma$ be a finite group acting on $E$ and fixing the zero section. 
Given a curve $C/k$ with a free $\Gamma$ action, we can form
\begin{equation}\label{Jacobian-exp} X = (E \times C)/\Gamma, \end{equation}
where the action on the product is diagonal. The quotient $X$ has a smooth Jacobian elliptic fibration $X \to C/\Gamma$.
%

We will classify deformations of elliptic fibre bundles over curves in several steps.  First, we study the Jacobian ones and prove that their deformations are always of the form (\ref{Jacobian-exp}).

Next, we study the relation between Jacobian and the non-Jacobian bundles. For an arbitrary elliptic fibre bundle $f \colon X \to C$ denote by $\Fib_{X/C}$ the functor of deformations of $X$ extending its fibration structure.
Associated to $X$ there is a Jacobian bundle over the same base, which we denote by $J$. Let $\Jac_{J/C}$ be the subfunctor of $\Fib_{J/C}$ of deformations with section. For precise definitions of these functors see \defref{functors}.

There is a natural map $\Fib_{X/C} \to \Jac_{J/C}$ given by taking the identity component of the relative Picard scheme. 

\begin{nnthm}[\ref{def-functor-mapping}]
 The map of deformation functors $\Fib_{X/C} \to \Jac_{J/C}$ is formally smooth and it holds
\[\dim (\Fib_{X/C}(k[\epsilon])) = \dim (\Jac_{X/C}(k[\epsilon])) + h^1(C, \Lie(J/C)). \]
\end{nnthm}

This resembles situation of elliptic fibrations over a field: The Jacobian fibrations can be dealt with explicitly and the non-Jacobian ones are described by a cohomological theory, based on the group structure of the former.

Next, we answer the question whether, given an elliptic bundle $X \to C$, every deformation admits an extension of the fibration structure:
\begin{nnthm}[\ref{kodaira-one}]
 If $X$ is of Kodaira dimension one, then the unique elliptic fibration extends to every deformation.
\end{nnthm}

Now we can address the liftability question: In section \ref{non-liftable-elliptic-surfaces} we will construct Jacobian elliptic fibre bundles over fields of characteristic two and three, which cannot be lifted as Jacobian bundles. By the above theorems, this is enough to show non-liftabilty.

\begin{nnthm}[\ref{non-lift}]
 There exist elliptic fibre bundles in characteristic two and three, that do not lift to characteristic zero.
\end{nnthm}

Another result is that proving liftability for a certain class of isotrivial elliptic fibrations would imply the Oort conjecture: Given a curve $C$ of higher genus, and a cyclic subgroup $G$ of $\Aut(C)$, then the pair $(C, G)$ can be lifted.
Our result is:

\begin{nnthm}[\ref{thm-oort-conj}]
 Choose an elliptic curve $E$ and let $G$ act on $E$ by translation.
Then the quotient $X =  (E \times C)/G$ is liftable if and only if the pair $(C, G)$ is liftable.
\end{nnthm}

As a further application of the theory, we treat the case of bielliptic surfaces. Here we need methods from the deformation theory of abelian schemes. 
Our main result is:

\begin{nnthm}[\ref{bielliptic-deformations}]
If $X$ is a bielliptic surface, then both elliptic fibrations extend under deformations. In other words, every deformation of a bielliptic surface is bielliptic.
\end{nnthm}

 For small $p$ one encounters phenomena, which do not appear when considering the same class of surfaces in characteristic zero.
For example, deformations become obstructed which was already observed by W. Lang in \cite{LangEx}.  There is also the possibility of deforming a Jacobian bielliptic surface into a non-Jacobian one, which is absent in characteristic zero.

\subsection*{Acknowledgments}
\thanks{This article is a part of my Ph.D. thesis. I am indebted to my advisor Stefan Schr\"oer for introducing me to liftabilty questions and helpful discussions. I would also like to thank William Lang, Christian Liedtke, Matthias Sch\"{u}tt and Philipp Gross and the referee for inspiring conversations, hints, and pointing out mistakes.}

\section{Preliminaries}
In this section we introduce some standard techniques which will be used later on.
Concerning deformation theory, we follow Schlessinger's fundamental paper \cite{SCHL1} in terminology, and freely make use of basic facts about pro-representable hulls of deformation functors.

A key problem that will appear in Sections \ref{deformations-kodaira-one}  and \ref{sec-bielliptic} is of the following form:
Given a deformation $\mdl X$ of some scheme $X$, what properties and additional structures carry over to $\mdl X$?
One example for such properties are \'etale coverings.


\begin{thm}[Theorem 5.5 and Theorem 8.3 \cite{sga1}]
Let $\mdl S$ be a scheme with a closed subscheme $S_0$ having the same topological space as $\mdl S$ itself.
Then the functor
\[ \mdl X \mapsto \mdl X \times_{\mdl S} S_0 \]
form the category of \'etale $\mdl S$-schemes to the category of \'etale $S_0$-schemes is an equivalence of categories.
\end{thm}

This theorem can be seen as a geometric form of Hensel's Lemma from commutative algebra.
We note two special cases: The categories of \'etale Galois covering of $\mdl X$ and $X_0$ are equivalent, and so are the categories of finite \'etale group schemes.
Recall that an \'etale covering $S' \to S$ is called Galois with group $G$ if $G$ acts on $S'$ as an $S$-scheme and we have an isomorphism 
\begin{equation*}
 G \times S' \simeq S' \times_S S' \;\text{  given by  }\; (\sigma, x) \mapsto (\sigma(x), x).
\end{equation*}

\subsubsection{Notations}\label{notations}
Finally, let us fix some notations.
By $k$ we denote an algebraically closed field of characteristic $p > 0$.
We denote by $W = W_\infty(k)$ its ring of Witt vectors. Let $\Alg_W$ be the category of local Artinian $W$-algebras having residue field $k$.
Every scheme and every ring will be assumed Notherian, and by a curve over some base scheme $S$ we will always mean a proper and connected one-dimensional $S$-scheme.

\begin{defn}\label{defn-efbundle}
 Let $S$ be a scheme over some ring $R$. An $R$-morphism $X \to S$ is called \emph{elliptic fibre bundles}
If there exists a surjective \'etale morphism $U \to S$ and an elliptic curve $E$ over $R$ such that
\[ U \times_S X \simeq U \times_{\Spec(R)} E. \]
\end{defn}

\section{Deformations of Jacobian elliptic fibre bundles}\label{sec-jacobian}
 A \emph{Jacobian} elliptic fibre bundle is a pair $(J/S, \epsilon)$ where $J \to S$ is an elliptic fibre bundle and $\epsilon$ is a section of $J \to S$.
We can consider $(J/S, \epsilon)$ as an elliptic curve over $S$.
In particular, there exists a unique commutative group scheme structure on $J/S$.

We are going to work over an algebra $\Lambda \in \Alg_W$.

%

\begin{prop}\label{prop-global-splitting}
Let $\mdl S$ be a proper flat $\Lambda$-scheme such that $\mdl S \otimes_\Lambda k$ is integral.
Let $\mdl J / \mdl S$ be an elliptic curve over $\mdl S$.
For some integer $n \ge 3$ which is prime to $p$, assume that the $n$-torsion subgroup scheme of $\mdl J$ is split i.e.,
there is an isomorphism \[  \mdl J[n] \simeq (\Z/n\Z)^2 . \]
Then there exists an elliptic curve $\mdl E$ over $\Lambda$ such that $\mdl J$ is isomorphic to $\mdl E \times_\Lambda \mdl S$.
\end{prop}
\begin{proof}
We can choose a level-$n$-structure on $\mdl J/\mdl S$.
By \cite[Corrolary 4.7.2]{KM} we obtain a morphism $c \colon \mdl S \to \MGamma{n}$ such that $\mdl J \simeq c^*(\mdl E^{univ})$ where $$(\mdl E^{univ}, \gamma \colon \mdl E^{univ}[n] \simeq (\Zm{n})^2)$$ is the universal family of the moduli problem.


Again by \cite{KM} we know that $\MGamma{n}$ is affine.
Thus $c$ factors over the affine hull of $\mdl S$, namely $\Spec(H^0(\mdl S, \O_{\mdl S})) = \Spec(\Lambda)$.
Therefore $\mdl J$ is just the pullback of an elliptic curve $\mdl E$ over $\Lambda$.
\end{proof}

In particular, we see that $\mdl J/\mdl S$ is an elliptic fibre bundle under the assumptions of \proref{prop-global-splitting}.
This can be generalized, because for an arbitray elliptic curve $\mdl J/\mdl S$, and an interger $n$ prime to $p$,
we always have that $\mdl J[n]$ is a finite \'etale group scheme over $\mdl S$, so 
there exists an \'etale Galois covering $\mdl S' \to \mdl S$ such that $\mdl J[n] \times_{\mdl S} \mdl S' \simeq (\Zm{n})^2$.

\begin{prop}\label{prop-global-structure}
Let $\mdl S$ be a proper flat $\Lambda$-scheme such that the special fibre $\mdl S \otimes_\Lambda k$ is regular.
Let $\mdl J / \mdl S$ be an elliptic curve over $\mdl S$, and
let $\mdl S' \to \mdl S$ be a finite \'etale Galois covering with group $G$ such that $\mdl J[n] \times_{\mdl S} \mdl S' \simeq (\Z/n\Z)^2$ for some $n \ge 3$. 
Then $$ \mdl J \simeq (\mdl E \times_\Lambda \mdl S')/G,$$
 where $\mdl E$ is an elliptic curve over $\Lambda$, and the action is the diagonal action given by the Galois action on $\mdl S'$ and by a homomorphism $ G \to \Aut(\mdl E)$ on the left factor. 
\end{prop}
\begin{proof}
The scheme $\mdl S'$ is connected, and because $\mdl S \otimes_\Lambda k$ is regular, so is $\mdl S' \otimes_\Lambda k$. In particular, $\mdl S' \otimes_\Lambda k$ is integral.
Hence the elliptic curve $$\mdl J \times_{\mdl S} \mdl  S' \to \mdl S'$$ satisfies the assumptions of \proref{prop-global-splitting}.
Thus there exists an elliptic curve $\mdl E$ over $\Lambda$  and an isomorphism $\mdl J \times_{\mdl S} \mdl S' \simeq \mdl E \times_\Lambda \mdl S'$.
In other words, we know that $\mdl J$ and $\mdl E \times_\Lambda \mdl S'$ are twists of each other, becoming isomorphic under the base change $\mdl S' \to \mdl S$. 

Twists of $\mdl E \times_\Lambda \mdl S$ are classified up to isomorphism by the Galois cohomology set
$ H^1(G, A(\mdl S')), $
where $A$ is the group scheme $\Aut(\mdl E \times_\Lambda \mdl S)$ and we consider its $\mdl S'$-valued points as Galois module under $G$.

We claim that the Galois action on $A$ is trivial:
We have a closed immersion $A \subset \Aut(\mdl E[n] \times_\Lambda \mdl S)$ by rigidity \cite[Corollary 2.7.3]{KM}.
However, since $\Lambda$ is a strict henselian ring, we find that $\mdl E[n] \times_\Lambda \mdl S$ is the constant group scheme $(\Z/n\Z)^2$ on $\mdl S$ which in turn implies that $\Aut(\mdl E[n] \times_\Lambda \mdl S)$ is the constant group scheme $\GL(2, \Z/n\Z)$ on $\mdl S$.

Finite \'etale group schemes over $\mdl S$ correspond to finite abstract groups with a continuous $\pi_1(\mdl S)$-action.
We saw that $A$ can be embedded into a group scheme with trivial $\pi_1(\mdl S)$-action, hence the action on $A$ has to be trivial as well.
The action of $G$ on $A$ is an induced action of a finite quotient $\pi_1(\mdl S) \twoheadrightarrow G$, and therefore trivial as well.
Thus we have
$$ H^1(G, \Aut(\mdl E \times_\Lambda \mdl S)(\mdl S')) \simeq \Hom(G, \Aut(\mdl E \times_\Lambda \mdl S)(\mdl S')). $$
For a homomorphism $\rho$ in the above group, the corresponding twist looks like $(\mdl E \times_\Lambda \mdl S')/G,$ where the action of $\sigma \in G$ is given by
\[
 (x, y) \mapsto (\rho(\sigma)(x), \sigma y). \qedhere
\] 
\end{proof}

Now, let $S \simeq \mdl S \otimes_R k$ denote the reduction of $\mdl S$.
Given an elliptic curve $E/S$ we can use the above results
to give a necessary and sufficient criterion for the existence of Jacobian liftings:

\begin{cor}\label{cor-obstructions-Jacobian-case}
 Let $J$ be an elliptic curve over $S$, given by $(E \times_k  S') / G$ for some \'etale Galois covering $ S' \to S$ with group $G$ (\proref{prop-global-structure}). Denote the action of $G$ on $E$ by $\rho_0$.
Then there exists a lifting $\mdl J \to \mdl S$ if and only if there exists a lifting $\mdl E$ of $E$ over $\Lambda$ together with an extension of the action $\rho_0$.
\end{cor}
\begin{proof}
To show sufficiency is easy. The covering $S' \to S$ lifts uniquely to $\mdl S' \to \mdl S$ which is again Galois with group $G$.
If a lifting $\mdl E$ of $E$ with the prescribed properties exists, simply put $\mdl J = (\mdl E \times_\Lambda \mdl S') /G$. This quotient will exist in the category of schemes because $G$ is finite.

In order to show necessity, assume that we have a lifting $\mdl J \to \mdl S$. Like before, we also have the unique lifting $\mdl S' \to \mdl S$ of the Galois covering.
Observe that $\mdl J[n] \times_{\mdl S} \mdl S'$ is split, since $\mdl J[n]$ is a finite \'etale group scheme and the reduction is split by assumption.
Using  \proref{prop-global-structure}, we find that $\mdl J \simeq (\mdl E \times_\Lambda \mdl S') / G$, where the action of $G$ on $\mdl E$ is denoted by $\rho$.
We claim that $\rho$ lifts the action $\rho_0$:

Consider the induced action of $\rho$ on $\mdl E[n]$ for some integer $n$. The categories of \'etale group schemes over $k$ and $R$ are equivalent, hence $\rho$ is determined by its action on the reduction $E[n]$.

For $n \ge 3$ we know that the group homomorphisms, given by restricting the automorphism group of an elliptic scheme to its $n$-torsion is injective \cite[Corollary 2.7.2]{KM}.
However the isomorphism type of $J[n]$ allows to read of the action of $G$ on $J[n]$, for it is given by a class in
$$ H^1(G, \Aut(J[n])(S')) \simeq \Hom(G,  \Aut(E[n])(S')) $$
and the element of the latter group which corresponds to $J[n]$ is just $\rho_0$.
Hence the restriction of $\rho$ to the reduction has to be $\rho_0$.
\end{proof}

\subsection{Non-liftable elliptic surfaces}\label{non-liftable-elliptic-surfaces}

We postpone the development of the general theory at this point to give some specific examples of Jacobian elliptic fibre bundles which are non-liftable.

\subsubsection*{Characteristic three}

For the first example, let $k$ be an algebraically closed field of characteristic three, and let $E$ be an elliptic curve over $k$ with $j$-invariant 0. By \cite[Appendix A, Proposition 1.2]{S1} the automorphism group $G$ of $E$ is a semidirect product $\Z/3\Z \rtimes \Z/4\Z$ where $\Z/4\Z$ acts on $\Z/3\Z$ in the unique non-trivial way.

As we shall see later on, there exists a smooth curve $C$ over $k$ such that there is a surjection $\pi_1(C) \twoheadrightarrow G$. Denote by $C' \to C$ the associated finite \'etale Galois cover. Now we set
$$ J = (E \times_k C')/G, $$
where the action of $G$ on $E$ is the action of the automorphism group.

\begin{lemma}\label{non-lift-autos}
Let the characterstic of $k$ be three, and
let $\Lambda$ be in $\Alg_W$. If for an elliptic curve $\mdl E$ over $\Lambda$ the order of the automorphism group of $\mdl E$ is greater than six, it follows $3 \cdot \Lambda = 0$.
\end{lemma}
\begin{proof}
Assume by contradiction, that the order of $\Aut_0(\mdl E)$ is greater than six.
Since two is a unit, there is a Weierstra\ss{} equation for $\mdl E$ of the following form:
$$ y^2 = x^3 + a_2x^2 + a_4x + a_6 $$
Admissible transformations look like $x \mapsto u^2x + r$ and $y \mapsto u^3y + u^2sx + t$.
The specific form of the equation implies $t = 0$ and $s = 0$.
Standard arguments show that either $u^4 = 1$ or $u^6 = 1$. Thus, an automorphism group of order greater than six would have to contain an element of the form $x \mapsto x + r$.

We get an equation $a_2 = a_2 + 3r$, which implies $3r = 0$. But $r$ has to be a unit, for otherwise the reduction map would not be injective on the automorphism group. Thus $3 = 0$ follows.
\end{proof}

Now we get as a direct consequence of \corref{cor-obstructions-Jacobian-case}:

\begin{prop}\label{pro-non-lift-jacobian-3}
The elliptic bundle $J$ can be lifted (as Jacobian fibration) only over rings in which $3 = 0$ holds.
\end{prop}

\subsubsection*{Characteristic two}
Now assume $k$ is a field of characteristic two.
Given an elliptic curve $E$ over $k$ with $j$-invariant 0, the group of automorphisms will be a semidirect product $G = Q \rtimes \Z/3\Z$, where $Q$ is the quaternion group. Similarly to \lemref{non-lift-autos}, one shows that neither $G$ nor $Q$ can lift to rings with $2 \neq 0$.
Now assume the existence of two curves $C_G$ and $C_Q$ together with  \'etale Galois covers $C_G' \to C_G$ with group $G$ and $C_Q' \to C_Q$ of group $Q$ respectively.

This gives rise to two Jacobian bundles: $J_G \simeq (C_G' \times E) / G$ and $J_Q \simeq (C_Q' \times E) / Q$. Again by \corref{cor-obstructions-Jacobian-case} it follows:

\begin{prop}\label{pro-non-lift-jacobian-2}
The elliptic bundles $J_G$ and $J_Q$ can be lifted (as Jacobian fibrations) only over rings in which $2 = 0$ holds.
\end{prop}

To finish this discussion, we have to establish the existence of curves with specific \'etale Galois coverings.

To this end, we use a powerful theory which is developed in \cite{PaSt}. First we fix some group theoretic invariants.
Let $G$ be a finite group with the property that the maximal $p$-Sylow subgroup $P$ is normal. We set $H = G/P$.
Then one can write $G$ as a semidirect product $P \rtimes H$.

We denote by $\cat P$ the maximal elementary abelian quotient of $P$, and consider it as a $\F_p$-vector space, which is possible since it is a $p$-torsion group. Let $Z(H)$ be the set of irreducible characters with values in $k$, and let $V_\chi$ be an irreducible $k$-representation of $H$ with character $\chi$.
On $P$, we have an $H$ action coming from the structure of the semidirect product. This induces an $H$-representation on $\cat P$.
Since $H$ is of order prime to $p$, this representation is semisimple, and we write
 $$\cat P \otimes_{\F_p} k \simeq \oplus V_\chi^{m_\chi}.$$
The $m_\chi$ are thus numerical invariants of the group $G$.

\begin{nnthm}[Theorem 7.4 \cite{PaSt}]
Let $G$ be a group having a normal $p$-Sylow subgroup $P$. Suppose $H = G/P$ is abelian. Then there exists a curve of genus $g \ge 2$ having an \'etale Galois covering with group $G$ if the minimal number of generators of $H$ is less than or equal than to $2g$, and $m_\chi \le g -1$ holds for every $\chi \in Z(H)$. 
\end{nnthm}

In the characteristic three example we had $G = \Z/3\Z \rtimes \Z/4\Z$. The minimal number of generators of $H = \Z/4\Z$ is one, the representation of $H$ on $\cat P$ is obviously irreducible and given by the sign involution. Thus the assumptions of the theorem are satisfied for some curve of genus $2$.

In the characteristic two examples we also have that the maximal $p$-Sylow group is normal. Thus for $g$ sufficiently large, we will find curves with the required coverings.

Later on, we will prove that $J$, $J_G$ and $J_Q$ are non-liftable even if we drop any additional requirements on the liftings.
This is done in two steps:
First, we prove that a lifting in the category of elliptic fibre bundles exists if and only if a Jacobian lifting exists. 
This is the content of Section \ref{sec-deformation-of-torsors}.
The final step is to see that \emph{every} lifting of an elliptic fibre bundle of Kodaira dimension one is an elliptic fibre bundle. This will be done in Section \ref{deformations-kodaira-one}.

\section{Deformations of elliptic torsors}\label{sec-deformation-of-torsors}
We start with some general theory on deformations of torsors under smooth commutative group schemes.
This mainly rephrases \cite[Remarque 9.1.9]{SGA3.2}.

We work in the following setting:
Fix a small extension of algebras in $\Alg_W$
$$ 0 \to I \to \Lambda \to \Lambda_0 \to 0. $$
Let $\mdl S$ be a flat $\Lambda$-scheme and set $\mdl S_0 = \mdl S \otimes_{\Lambda} \Lambda_0$.
We have a closed immersion $i \colon \mdl S_0 \to \mdl S$.
Let $\mdl G$ be a smooth commutative $\mdl S$-group scheme and set $\mdl G_0 = \mdl G \times_{\mdl S}\mdl S_0$.

For a group functor $\mdl F$ on the category of $\mdl S_0$-schemes, we defined the pushforward functor $i_*\mdl F$ on $\mdl S$-schemes by sending a $\mdl S$-scheme $\mdl T$ to $\mdl F(\mdl T \times_{\mdl S} \mdl S_0)$.
There is a natural specialization map  $s \colon \mdl G \to i_*\mdl G_0$ of group functors.
To investigate its kernel, we introduce a coherent sheaf on $\mdl S_0$:
 $$\sheaf L = \Lie(\mdl G_0/\mdl S_0) \otimes_{\O_{\mdl S_0}}  I.$$
We have a sequence of group functors on $\mdl S$
\begin{equation}\label{specialization}
 0 \to i_* \sheaf L \to \mdl G \xrightarrow{s} i_*\mdl G_0 \to 0,
\end{equation}
whose exactness follows from the smoothness of $\mdl G_0$, as can be seen affine locally.
Taking \'etale cohomology of (\ref{specialization}), we obtain the fundamental long exact sequence
  \begin{multline}\label{fundamental-sequence}
   0 \to i_*\mdl G_0(\mdl S)/s(\mdl G(\mdl S)) \to H^1(\mdl S, i_* \sheaf L)  \to H^1(\mdl S, \mdl G)  \xrightarrow{s} \\ \to H^1(\mdl S, i_*\mdl G_0) \to H^2(\mdl S, i_* \sheaf L) .
  \end{multline}  
The sheaves $\mdl L$ and $i_*\mdl L$ are coherent modules. We find $$ H^i(\mdl S, i_* \sheaf L) \simeq H^i(\mdl S_0, \sheaf L) \simeq H^i_{zar}(\mdl S, \Lie(G_0/S_0) ) \otimes I. $$ 
Moreover, we claim that the group $H^1(\mdl S, i_*\mdl G_0)$ is isomorphic to $H^1(\mdl S_0,\mdl G_0)$:
By the Leray spectral sequence we get an exact sequence of \'etale cohomology groups
$$ 0 \to H^1(\mdl S, i_*\mdl G_0) \to H^1(\mdl S_0,\mdl G_0) \to H^0(\mdl S_0, R^1 i_*\mdl G_0), $$
with vanishing last term:
It is enough to show that $(R^1 i_*\mdl G_0)_x = 0$ for every closed point $x$ of $\mdl S$; i.e. of $\mdl S_0$.
By \cite[Theorem 1.15]{Mil1} it follows that $(R^1 i_*\mdl G_0)_x \simeq H^1(\Spec(\O_{{\mdl S_0},x}^\wedge),\mdl G_0)$ and the last group vanishes since $\Spec(\O_{{\mdl S_0},x}^\wedge)$ has only  the trivial \'etale covering. Here, we make use of the assumption that $k$ is algebarically closed.

In our situation, this means the following: Let $J / C$ be a Jacobian elliptic fibre bundle over a curve $C$ over $k$.
Let $\mdl J_0 / \mdl C_0$ be a Jacobian lifting of $J/C$ over $\Lambda_0 \in \Alg_W$. 

\begin{prop}\label{deforming-elliptic-torsors}
 Let $\mdl J \to \mdl C$ be a Jacobian lifting of $\mdl J_0/\mdl C_0$ to $\Lambda$. Then every $\mdl J_0$-torsor $\mdl X_0$ over $\mdl C_0$ lifts to a $\mdl J$-torsor over $\mdl C$.
Furthermore, let $m$ be an integer prime to $p$. Then the restriction map
\[
 H^1(\mdl C, \mdl J)[m] \xrightarrow{\sim} H^1(\mdl C_0, \mdl J_0)[m]
\]
is bijective. This means that liftings of torsors are unique up to $p$-torsion.
\end{prop}
\begin{proof}
From the sequence (\ref{fundamental-sequence}) we know that the obstruction to lifting the cohomology class associated to
 $\mdl X_0$ lies inside $H^2(\mdl C_0, \Lie(\mdl J_0/\mdl C_0)) \otimes I$. Note that since $\Lie(\mdl J_0/\mdl C_0)$ is a coherent $\O_{\mdl C_0}$-module, we can compute its cohomology with respect to the Zariski topology.
Since $\mdl C_0$ is one-dimensional, this group is zero.

Once we have lifted the cohomology class, we have to answer the question, whether it is associated to a representable $\mdl J$-torsor.
By \cite[Lemme XIII]{Raynaud119} below, this will be the case if it is torsion.
We claim that $H^1(\mdl C, \mdl J)$ is torsion:
Since $H^1(C,J)$ is torsion, it is enough to show that $H^1(\mdl C_0, \Lie(\mdl J_0/\mdl C_0)) \otimes I$ is torsion, then the assertion will follow by induction. However the former group is a $\Lambda$-module, and $\Lambda$ itself is annihilated by some power of $p$.

The second statement follows now directly by taking $m$-torsion in (\ref{fundamental-sequence}).
\end{proof}

\begin{rem}
 In the case where the base is zero dimensional, one recovers the well known fact that the Tate-{\v{S}}afarevi{\v{c}} group of an elliptic curve over a complete local ring with algebraically closed residue field is zero, since the first cohomology of the Lie algebra vanishes.
\end{rem}

We want to rephrase \proref{deforming-elliptic-torsors} in the language of deformation functors.
For that purpose, we define two deformation functors associated to an elliptic fibre bundle $X \to C$ over $k$.

\begin{defn}\label{functors}
By a deformation of $X$ over some $\Lambda \in \Alg_W$, we mean a pair $(\mdl X, \epsilon)$, where $\mdl X$ is a flat scheme over $\Spec(\Lambda)$ and $\epsilon$ is an isomorphisms $\epsilon \colon \mdl X \otimes_\Lambda k \simeq X$. 
Let $\Def_{X} \colon \Alg_W \to (Sets)$ denote the functor, which sends $\Lambda \in \Alg_W$ to the set of isomorphism classes of deformations of $X/C$.

By a \emph{deformation of a fibration} $X/C$, we understand a deformation $(\mdl X, \epsilon)$ of $X$ together with a map $\mdl X \to \mdl C$, such that the isomorphism $\epsilon$ is in fact an isomorphism of $C$-schemes.
The functor of deformations of $X$ as fibration is denoted by $\Fib_{X/C} \colon \Alg_W \to (Sets)$.

Two deformations of $\mdl X/\mdl C$ and $\mdl X'/\mdl C$  are called \emph{isomorphic} if 
there exists an isomorphisms of deformations, which is also an isomorphism of $\mdl C$-schemes.

Let $(J/C, e_0)$ be a Jacobian elliptic fibre bundle. 
We define the functor $\Jac_{J/C}$ by sending $\Lambda \in \Alg_W$ to a pairs $(\mdl J/\mdl C, e)$, where $\mdl J/\mdl C$ is a deformation of $J/C$ and $e$ is a lift of $e_0$.
For an element $(\mdl J/\mdl C, e)$ of $\Jac_{J/C}$ notice that a different choice of $e$ leads to an isomorphic element of $\Jac_{J/C}$. Thus we view $\Jac_{J/C}$ as a subfunctor of $\Fib_{J/C}$, i.e., the subfunctor of those deformations admitting a section. 
\end{defn}

We get a natural map $\Fib_{X/C} \to \Jac_{J/C}$ as follows:
For a deformation $\mdl X / \mdl C$ (not necessarily having a section)
we consider the identity component of its Picard scheme.
Since $k$ is of characteristic $p$, we can always lift an appropriate $p$-th power of a relativly ample line bundle of $X \to C$.  Therefore the representabilty of $\Pic_{X/C}$ follows from:

\begin{nnthm}[Theorem 4.8 \cite{Picard}]
 Let $Z$ be a projective flat $S$-scheme, having integral geometric fibres.
Then $Pic_{Z/S}$ is representable by a separated $S$-scheme.
\end{nnthm}

Since $\mdl X/\mdl C$ is an elliptic fibre bundle, the Picard scheme will be smooth, and the identity component $\Pic^0_{\mdl X/\mdl C}$ is a smooth elliptic scheme over $\mdl C$. 
Now $\mdl X/\mdl C$ is in a natural way a torsor under $\Pic^0_{\mdl X/\mdl C}$
coming from the isomorphism
$$\Pic^1_{\mdl X/\mdl C} \simeq \mdl X/\mdl C. $$
We define the natural map
$ \Fib_{X/C} \to \Jac_{J/C} $
by sending the fibration $\mdl X/\mdl C$ to $\Pic^0_{\mdl X/\mdl C}$.
In this language, \proref{deforming-elliptic-torsors} now becomes the first part of our main theorem:

\begin{thm}\label{def-functor-mapping}
The map of functors $\Fib_{X/C} \to \Jac_{J/C}$ is formally smooth. Moreover, we have
$$ \dim (\Fib_{X/C}(k[\epsilon])) = \dim (\Jac_{X/C}(k[\epsilon])) + h^1(C, \Lie(J/C)). $$
\end{thm}
\begin{proof}
Recall that $\Fib_{X/C} \to \Jac_{J/C}$ is formally smooth if for any surjection $\Lambda \twoheadrightarrow \Lambda_0$ in $\Alg_W$ the induced map
\begin{equation*}
 \Fib_{X/C}(\Lambda) \to \Fib_{J/C}(\Lambda_0) \times_{\Jac_{J/C}(\Lambda_0)} \Jac_{J/C}(\Lambda)
\end{equation*}
is surjective. By induction it suffices to verify this for small extensions. However, this follows directly from \proref{deforming-elliptic-torsors} applied to every element $\mdl J$ of $\Jac_{J/C}(\Lambda)$ with reduction $\mdl J_0$ over $\Lambda_0$ and a $\mdl J_0$-torsor $\mdl X_0$.

To prove the statement about the tangent space dimensions, first note that $\Fib_{X/C}$ fulfills the Schlessinger criteria and carries therefore a vector space structure on its tangent space.
We are going to determine the kernel of the following linear map
$$ \Fib_{X/C}(k[\epsilon]) \to \Jac_{J/C}(k[\epsilon]), $$
which consists of torsors under $ J \otimes k[\epsilon]$. To determine this group, we use again (\ref{fundamental-sequence}). The first term vanishes, since every section $C \to J$ lifts to the trivial deformation. Hence the kernel is given by 
$H^1(C, \Lie(J/C)) \otimes I$. 
\end{proof}

\section{Elliptic fibre bundles of Kodaira dimension one}\label{deformations-kodaira-one}
We sum up some facts about invertible sheaves associated to elliptic fibre bundle.

\begin{lemma}\label{lem-inv-sheaves}
 Let $f \colon X \to C$ be an elliptic fibre bundle over a smooth proper curve $C$ over $k$.
\begin{enumerate}
 \item The sheaf $\sheaf L = R^1f_*\O_X$  on $C$ is invertible and a torsion element in $\Pic(C)$.
 \item The relative tangent bundle $\Theta_{X/C}$ is isomorphic to $f^*\sheaf L$.
 \item The Kodaira dimension of $X$ equals the Kodaira dimension of $C$.
\end{enumerate}
\end{lemma}
\begin{proof}
We prove (i).  To see that $\sheaf L$ is invertible just note that $X \to C$ has no multiple and hence no wild fibers.
To prove that $\sheaf L$ is torsion we proceed as follows:
 Let $g \colon J \to C$ be the Jacobian of $X$, and dentote by $e \colon C \to J$ the zero section.
 Since $J = \Pic^0(X/C)$ we have
\[\sheaf L \simeq \Lie(J/C) \simeq e^*\Theta_{J/C} \]
(see \cite[Proposition 1.3]{RaynaudMultipleFibers}). Thus it suffices to show that $\Theta_{J/C}$ is torsion in $\Pic(J)$.
%
By \proref{prop-global-structure}, 
we know there exists an \'etale covering $C' \to C$ and an elliptic curve $E$ over $k$ such that
$J \times_C C'$ is isomorphic to  $E \times_k C'$ which we denote by  $J'$.
Let $q$ be morphism $J' \to J$ given by base change.
Since $q$ is \'etale, it follows $q^*\Theta_{J/C} \simeq \Theta_{J'/C'} \simeq \O_{J'}$.
Now, we apply the norm map $N \colon \Pic(J') \to \Pic(J)$ (see \cite[6.5]{ega2}), associated to the morphism $q$:
\[ \O_J \simeq N(q^*\Theta_{J'/C'}) \simeq \Theta_{J/C}^{\otimes d} \]
where $d$ is the degree of $q$. This proves the claim.

For (ii) note that the canonical bundle formula for $X \to C$ reads
\[ \omega_X \simeq f^*(\sheaf L^{-1} \otimes \omega_C). \]
The claim follows form the expression
 $$(\Theta_{X/C})^{-1} \simeq \Omega^1_{X/C} \simeq \omega_X \otimes (f^*\omega_C)^{-1}.$$

Finally (iii) follows from (i) and the canonical bundle formula.
\end{proof}

In particular, we see that an elliptic fibre bundle is of Kodaira dimension one, if and only if the base curve is of genus greater than or equal to 2.
It is a general fact form the theory of elliptic surfaces, that on a surface of Kodaira dimension one, there exists exactly one elliptic fibration, given by a suitable power of the canonical sheaf.
We generalize this fact to deformations, by showing that the unique fibration lifts to an arbitrary deformation of the total space and that this lifting is unique. In other words:

\begin{thm}\label{kodaira-one}
 For an elliptic fibre bundle $f\colon X \to C$ of Kodaira dimension one, the forgetful map of deformation functors $\Fib_{X/C} \to \Def_{X}$ is an isomorphism.
\end{thm}

First, we show injectivity:

\begin{prop}\label{pro-fibration-is-canoical}
 Let $f \colon \mdl X \to \mdl C$ be a deformation over $\Lambda \in \Alg_W$ of an elliptic fibre bundle $f \colon X \to C$.
Then $f$ is defined by a suitable power of the canonical sheaf $\mdl \omega_{\mdl X/\Lambda}$.
In particular, $f$ is unique.
\end{prop}
\begin{proof}
 By smoothness of $f$, we get an exact sequence
\[ 0 \to f^*\omega_{\mdl C /\Lambda} \to \Omega^1_{\mdl  X / \Lambda} \to \omega_{\mdl X/\mdl C} \to 0. \]
The outer terms are invertible sheaves. It follows $\omega_{\mdl X/\Lambda} = f^*\omega_{\mdl C/\Lambda} \otimes \Omega^1_{\mdl X/\mdl C}$.

We claim that $\Omega^1_{\mdl X/\mdl C}$ is a torsion element in $\Pic(\mdl X)$.
Let $f_0 \colon X \to C$ denote the reduction of $f$. By \lemref{lem-inv-sheaves} (i) we know that, the invertible sheaf
 $\sheaf L = R^1 f_{0*} \O_X$ is a torsion element in $\Pic(C)$.
The same is true for $\Omega^1_{X/C}$ because $(f_0^* \sheaf L)^\wedge \simeq \Omega^1_{X/C}$ (\lemref{lem-inv-sheaves} (ii)). 
We conclude that $\Omega^1_{\mdl X/\mdl C}$ is of finite order by induction on the length of $\Lambda$:
The tangent space of the deformation functor of invertible sheaves is $H^1(X, \O_X)$, which is a $p$-torsion group, since $k$ is of characteristic $p$.
So let $n$ denote the order of $\Omega^1_{\mdl X/\mdl C}$. For any integer $m$ we get $$\omega_{\mdl X/\Lambda}^{\otimes mn} \simeq f^*(\Omega^1_{\mdl C/\Lambda})^{\otimes mn}.$$
Since $\Omega^1_{\mdl C/\Lambda}$ is ample on $\mdl C$, we find that $\omega_{\mdl X/\Lambda}$ is semiample on $\mdl X$.
Furthermore, $\Proj(\Sym(\omega_{\mdl X/\Lambda}^{\otimes n}))$ is isomorphic to 
$\Proj(\Sym(\Omega^1_{\mdl C/\Lambda})^{\otimes n}) \simeq \mdl C$.
This follows because $H^0(\mdl X, f^*(\omega_{\mdl C/\Lambda})^{\otimes mn}) \simeq H^0(\mdl C, (\omega_{\mdl C/\Lambda})^{\otimes mn})$ by the projection formula.

Denote the map $\mdl X \to \mdl C$ given by the canonical sheaf by $f_{can}$.
On the reduction we find $f \otimes_\Lambda k = f_0 = f_{can} \otimes_{\Lambda} k$.
However, liftings of $f_0$ are unique. It is enough to prove this for first order deformations.
The tangent space of the functor of deformations of $f_0 \colon X \to C$ is $H^0(X', f_0^*\Theta_{C})$ and this vanishes since the dual $f_0^*\omega_{C}$ of $f^*\Theta_{C}$ has non-zero global sections, because $g(C) \ge 2$.
\end{proof}

To prove the surjectivity of $\Fib_{X/C} \to \Def_X$ we need an estimate for $h^1(X, \Theta_{X})$:

\begin{lemma}\label{invariants-kodaira-one}
 Denote by $g \ge 2$ the genus of $C$, and set $\sheaf L = R^1 f_*\O_{X}$. We get
\begin{equation}\label{invariants-inequality}
 h^1(X, \Theta_{X}) \le g - 1 + h^0(C, \sheaf L) + h^0(C, \sheaf L^2) + 3g - 3. 
\end{equation}
If $X$ is Jacobian, we get equality in (\ref{invariants-inequality}).
\end{lemma} 
\begin{proof}
 Since $f$ is smooth, we have an exact sequence
\begin{equation}\label{theta-sequence}
 0 \to \Theta_{X/C} \to \Theta_{X} \to f^*\Theta_{C} \to 0 .
\end{equation}
This gives rise to an exact sequence of cohomology groups
$$ H^1(X, \Theta_{X/C}) \to H^1(X, \Theta_{X}) \to H^1(X, f^*\Theta_{C}). $$
Thus $h^1(X, \Theta_{X}) \le  h^1(X, f^*\Theta_{C}) + h^1(X, \Theta_{X/C})$.

By \lemref{lem-inv-sheaves} we have $\Theta_{X/C} \simeq f^*\sheaf L$.
To compute $h^1(X, \Theta_{X/C})$ we use the Leray spectral sequence and the isomorphisms given by projection formula 
$(f_*\Theta_{X/C} \simeq f_*f^*\sheaf L \simeq \sheaf L$ and  $R^1 f_* \sheaf L \simeq \sheaf L^{\otimes 2})$ to obtain
$$ 0 \to H^1(C, \sheaf L) \to H^1(X, \Theta_{X/C}) \to H^0(C, \sheaf L^{\otimes 2}) \to 0. $$
By the Riemann-Roch Theorem we get $h^1(C, \sheaf L) = g - 1 + h^0(C, \sheaf L)$.
Thus
$$ h^1(X, \Theta_{X/C}) = g - 1 + h^0(C, \sheaf L) + h^0(C, \sheaf L^{\otimes 2}). $$

For $h^1(X, f^*\Theta_{C})$ we obtain with the same approach and using $\Theta_C \simeq f_*f^*\Theta_C$ the following sequence
$$ 0 \to H^1(C, \Theta_{C}) \to H^1(X, f^*\Theta_{C}) \to H^0(C, \sheaf L \otimes \Theta_{C}) \to 0. $$
Since $g > 1$, the last term vanishes and we get $h^1(X, f^*\Theta_{C}) = 3g - 3$.

In the Jacobian case let $s \colon C \to X$ denote the section. The natural map $f^*\Omega^1_C \to \Omega^1_X$ has a global splitting given locally by
$d (g) \mapsto d (s^*g) \otimes 1 $. Dualizing yields a splitting of (\ref{theta-sequence}).
%
\end{proof}

Now, we are set up to show the surjectivity of the inclusion $\Fib_{X/C} \to \Def_{X}$.

\begin{prop}\label{kodaira-one-prop}
Let $\Lambda$ be an object of $\Alg_W$.
Every deformation $\mdl{X} \in \Def_{X}(\Lambda)$ of the total space of $X$ admits a lifting of the fibration on $X$; in other words $\mdl{X} \in \Fib_{X/C}(\Lambda)$.
\end{prop}
\begin{proof}
Denote by $J$ the Jacobian of $X$. 
By \proref{prop-global-structure}, we know that there is an \'etale Galois covering $C' \to C$ with group $G$ such that $J' = J \times_{C} {C'} = E \times_k C'$, for some elliptic curve $E$ over $k$.
Since forming $\Pic^0$ commutes with base change, the Jacobian associated to the fibration $X' = X \times_{C} C'/C'$ will be $J'$.
We denote by $\mdl X' \to \mdl X$ the unique lifting of $X' \to X$.

We claim that $\mdl X'$ admits an elliptic fibration. To see this, we show that the deformation functors $\Fib_{X'/C'}$ and $\Def_{J'}$ are isomorphic.

The functor of Jacobian deformations of $J'$ is unobstructed (see \corref{cor-obstructions-Jacobian-case}), so we conclude by \proref{deforming-elliptic-torsors}, that $\Fib_{X'/C'}$ is unobstructed as well.
It remains to show that
$$ h^1(X', \Theta_{X'}) = \dim( \Fib_{X'/C'}(k[\epsilon])). $$

Let $g$ denote the genus of $C$.
We have $h^1(X', \Theta_{X'}) \le 4g - 2$ by  \lemref{invariants-kodaira-one}.
As for $\Fib_{X'/C'}(k[\epsilon])$, we have $(3g - 3) + 1$ dimensions coming from the functor of Jacobian deformations of $J'$: Namely $3g - 3$ from the deformations of the $C'$, and one dimension coming from $E$.
The first cohomology of the Lie algebra $\O_{C'}$ of $J'$ gives $g$ additional dimensions.

Now, we come back to our deformation $\mdl{X}$ of $X$.
By \proref{pro-fibration-is-canoical} the fibration $g \colon \mdl{X}' \to \mdl C'$ is defined in terms of the canonical bundle of $\mdl X'$.
The action of $G$ an $\mdl X'$ induces an action on the canoical model $\mdl C'$ of $\mdl X'$. For $\sigma \in G$ we get a diagram:
 \begin{equation*}
  \xymatrix{ \mdl X' \ar[d] \ar[r]^\sigma  & \mdl X' \ar[d] \\
             \mdl C' \ar[r]^\sigma          & \mdl C'}
 \end{equation*}
This implies that $g \colon \mdl X' \to \mdl C'$ descents to a fibration $f \colon \mdl X \to \mdl C$ on $\mdl X$.
%
\end{proof}

We are going to present some applications of \thmref{kodaira-one}.
Recall that in Section \ref{sec-jacobian}, we constructed a Jacobian elliptic fibre bundle $J$ over some curve of genus two over a field of characteristic 3, which was shown to be non-liftable as Jacobain elliptic fibre bundle.
We also gave two examples denoted by $J_G$ and $J_Q$, showing the same behaviour in characteristic two.

\begin{thm}\label{non-lift}
 The elliptic fibre bundles $J$ (in characteristic three) and $J_G, J_Q$ (in characteristic two) do not admit a formal lifting to characteristic zero. 
\end{thm}
\begin{proof}
We already saw that $J$, $J_G$ and $J_Q$ cannot be lifted as Jacobian elliptic fibre bundles (\proref{pro-non-lift-jacobian-3} and \proref{pro-non-lift-jacobian-2}).
From \proref{deforming-elliptic-torsors}, it follows that the same is true for liftings which are not Jacobian but admit an elliptic fibration.
Finally, observe that the base curves in both cases are of genus $g \ge 2$, which implies by \lemref{lem-inv-sheaves} (iii), that the Kodaira dimensions of $J$, $J_G$ and $J_Q$ are 1. Now, by \thmref{kodaira-one}, we get that every deformation is elliptic.
\end{proof}

One can make an interesting remark here. Recall the following conjecture:

\begin{conj}[F. Oort, 1985]
 Let $k$ be a field of characteristic $p$, and let $C$ be a smooth curve of genus at least 2.
Let $G$ be a cyclic subgroup of  $\Aut(C)$. Then there exists a lifting to characteristic zero of the pair $(C, G)$.
\end{conj}

The conjecture is known to hold if the order of $G$ is not divided by $p^3$ (\cite{MR1645000}).
Given a curve $C$ and a cyclic subgroup $G \subset \Aut(C)$,  we construct an elliptic surface $X \to B =  C/G$, 
which has a formal lifting to characteristic zero if and only if the pair $(C, G)$ is liftable.

Let $E$ be an ordinary elliptic curve over $k$. An action of $G$ on $E$ is given by choosing a point of order $\ord(G)$. We set $X = (E \times C)/G$, where we divide out by the diagonal action. Note that the action of $G$ on the product is free. The surface $X$ will have an elliptic fibration coming from the projection $E \times C \to C$.
However, this fibration will in general not define an elliptic fibre bundle, because the fixed points of the $G$ action on $C$ will give rise to multiple fibers.

\begin{thm}\label{thm-oort-conj}
 The surface $X$ is formally liftable to characteristic zero if and only if there exists a lifting $\mdl C$ of $C$ together with a lifting of $G$.
\end{thm}
\begin{proof}
 The ``if''-part is clear. Assume there exists a lifting $\mdl X$ of $X$ over some local Artinian ring $\Lambda$ with residue field $k$. Since the covering $X' = E \times C \to X$ given by the quotient map is \'etale, we find a lifting $\mdl X' \to \mdl X$ which is again Galois with group $G$.

Now, $X'$ is an elliptic fibre bundle of Kodaira dimension one, hence by \proref{kodaira-one-prop} we know that $\mdl X'$ comes with its canonical fibration $\mdl X' \to \mdl C$ lifting $X' \to C$.
Since $\mdl C$ is the canonical model of $\mdl X'$, we get an induced $G$ action on $\mdl C$, coming from the $G$-action on $\mdl X$. Since the automorphism scheme of a curve of higher genus is unramified, we conclude that the $G$ action on $\mdl C$ is faithful.
\end{proof}

\section{Deformations of bielliptic surfaces}\label{sec-bielliptic}
Let $X$ be minimal smooth surface over $k$,  of Kodaira dimension zero and with invariants $b_1 = b_2 = 2$.
Directly from the invariants, we get that the Albanese of $X$ is an elliptic curve.
The associated map $f \colon X \to \Alb(X)$ is either a smooth elliptic fibration (see \cite[Proposition 5]{BM2}) or a quasi-elliptic fibration. In the elliptic case, we call $X$ a bielliptic surface.

To keep the presentation streamlined, we first consider the cases where $f$ has a section. 
From \proref{prop-global-structure} we know that $X$ is given by a quotient
$$ (E \times F)/G $$
where $F \to \Alb(X)$ is an \'etale Galois covering with group $G$ (i.e. an \'etale isogeny) and $G$ acts on $E$ fixing the zero section.
Note that this action cannot be trivial, for otherwise $X$ would be an abelian surface.
Without loss of generality, we assume the action of $G$ on $E$ faithful.
Since the fundamental group of an elliptic curve is abelian, $G$ has to be abelian too.
It follows that $G$ equals $\Z/d\Z$, where $d \in \{2, 3, 4, 6\}$. 
We fix the image of a generator of $G$ in $\Aut_0(E)$, and denote it by $\omega$.
As a first step, we calculate some invariants of $X$ depending on $p$ and $d$ which are important for the deformation behaviour of $X$.
\begin{prop}\label{pro-bielliptic-invariants}
Let $X$ be a bielliptic surface. Denote by $f \colon X \to C$ its Albanese map. Assume that $f$ has as section.
 Write $$X = (E \times C')/G.$$
 If $d$ is not a power of $p$ and $d \neq 2$ we have
\begin{align*}
 h^0(\Theta_X) = 1, && h^1(\Theta_X) = 1, && h^2(\Theta_X) = 0, && h^1(C, \Lie(X/C)) = 0.
\end{align*}
If $d = 2$ and $p \neq 2$ we get
\begin{align*}
 h^0(\Theta_X) = 1, && h^1(\Theta_X) = 2, && h^2( \Theta_X) = 1, && h^1(C, \Lie(X/C)) = 0.
\end{align*}
Whereas if $d$ is a power of $p$ it holds
\begin{align*}
 h^0(\Theta_X) = 2, && h^1(\Theta_X) = 4, && h^2(\Theta_X) = 2, && h^1(C, \Lie(X/C)) = 1.
\end{align*}
Let $Y/C$ be a non-Jacobian bielliptic surface with Jacobian $X/C$.
Then we have $$h^i(Y, \Theta_Y) = h^i(X, \Theta_X).$$
\end{prop}
\begin{proof}
Since $f$ is smooth, we have an exact sequence
\begin{equation}\label{eq-differential-sheafs}
  0 \to \Theta_{X/C} \to \Theta_{X} \to f^*\Theta_{C} \to 0.
\end{equation}
Since the action of $\Zm{d}$ on $E \times C'$ is diagonal \eqref{eq-differential-sheafs} is split.
Therefore $\Theta_X$ decomposes as
\begin{equation}\label{decomposition}
 \Theta_{X} \simeq \Theta_{X/C} \oplus f^*\Theta_C. \end{equation}

Since $C$ is an elliptic curve,  $f^*\Theta_C \simeq \O_X$.
As for $\Theta_{X/C}$, it will be a torsion line bundle of order $\ell$ equal to the order of the induced action of $\Z/d\Z$ on $\Theta_{E}$.
To see this, note that $\Z/d\Z$ acts trivially on $\Theta_E^{\otimes \ell}$ because the induced action is by roots of unity. Thus a section of $\Theta_E^{\otimes \ell}$ will descend to a section of $\Theta_{X/C}^{\otimes \ell}$.
We claim that $\ord(\Theta_{X/C}) = \ell$ where $d = \ell p^n$ with $\ell$ prime to $p$. This is seen as follows:

The action on $\Theta_E$ is determined by the action on the $k$-vector space $\Lie(E)$.
If $d$ is a power of $p$, this action has to be trivial, since $\Hom(\Z/p^n\Z, \G_m) = 0$.
To determine the action in general, note that we have a subgroupscheme $H \subset E$ of height one,
such that $\Lie(H) = \Lie(H)$.
In fact, the total space of $H$ is given by $\Spec(\O_E/\ideal m_{0, E}^p)$ and $H$ is isomorphic either to $\mu_p$ or to $\alpha_p$.
Because of height one, the map given by the Lie functor $\Aut(H) \to \Aut(\Lie(H))$ is injective. In fact, it will be an isomorphism if we restrict to maps of $p$-Lie algebras (see \cite[Section 14]{AV}).

The group scheme $H$ is of rank $p$, so if $p > 4$ we get $\ord(\omega) = \ord(\omega|_H)$ by rigidity \cite[Corollary 2.7.3]{KM}.
If $p = 3$ and $d = 2$, we know that $\omega$ will act on $\Lie(E)$ as involution.
If $d = 4$ the same argument applies to $\omega^2$.

If $p = 2$ and $d = 3$ we have $H = \alpha_2$ since $j(E) = 0$. 
Assume $\omega$  induces the identity on $H$. Then the associated trace map $\tr_\omega = \Id + \omega + \omega^2$ would give multiplication by 3, which is an isomorphisms on $H$. However, we know that $\tr_\omega$ is the zero map on $E$ (see \lemref{trace-map} below).

Now, it is easy to calculate the invariants. 
Denote by $\epsilon \colon C \to X$ the zero section of $X$.
We have $\sheaf L = \Lie(X) = \epsilon^*\Theta_{X/C}$, and since $f^*\Lie(X) \simeq \Theta_{X/C}$ it follows that
$$ \ord(\Lie(X)) = \ord(\Theta_{X/C}). $$
The statement about the cohomology of $\Lie(X)$ follows, since it is a line bundle of degree zero and therefore
$$ h^1(C, \Lie(X)) = h^0(C, \Lie(X)).$$
However, the last term is not zero if and only if $\Lie(X)$ is trivial.
We also get that $$h^0(X, \Theta_X) = h^0(X, \Theta_{X/C}) + h^0(X, f^*\Theta_C)  = h^0(X, \Lie(X)) + 1.$$

To compute $h^1(X, \Theta_X)$ we treat both summands in \eqref{decomposition} separately. By \lemref{lem-inv-sheaves} (ii) we have $\Theta_{X/C} \simeq f^*R^1f_*\O_X \simeq f^*\Lie(X)$.
By the projection formula ($R^1 f_* f^* \sheaf L \simeq \sheaf L^{\otimes 2}$) and the Leray spectral sequence we get
$$ 0 \to H^1(C, \sheaf L) \to H^1(X, \Theta_{X/C}) \to H^0(C, R^1 f_*f^*\sheaf L) \to 0. $$
Thus we have:
\[
 h^1(X, \Theta_{X/C}) = \begin{cases} 0 \; \text{ if } \; \ell > 2 \\ 
                                      1  \; \text{ if } \; \ell = 2 \\ 
                                      2  \; \text{ if } \; \ell = 1 \end{cases}
\]
For $h^1(X, f^*\Theta_C)$ we obtain similarly:
$$ 0 \to H^1(C, \Theta_{C}) \to H^1(X, f^*\Theta_{C}) \to H^0(C, \sheaf L \otimes \Theta_{C}) \to 0 $$
Since $\Theta_C \simeq \O_C$ we find $h^1(X, f^*\Theta_C) = 2$ if $\sheaf L$ is trivial, and $h^1(X, f^*\Theta_C) = 1$ otherwise.

This proves the statement about $h^1(X, \Theta_X)$. To compute $h^2(X, \Theta_X)$ we just observe that $\chi(\Theta_X) = 0$, because $\chi(\Theta_{E \times C'}) = 0$.

The statement about a non-Jacobian bielliptic surface $g \colon Y \to C$ with Jacobian $X/C$ follows from the expression
\[  R^1 g_*\O_Y \simeq \Lie(X/C) \simeq \epsilon^*\Theta_{J/C} \]
and $g^*(\sheaf L) = \Theta_{Y/C}$.
\end{proof}

\subsection{The versal families}
Let $X = (E \times F)/G$ be a Jacobian bielliptic surface over $k$.
First, we study the deformation functor $\Jac_{X/C}$.

By \proref{prop-global-structure} we know the structure of Jacobian deformations of $X$.
They will be of the form $(\mdl E \times \mdl F)/\Gamma$.
Here, $\mdl E$ is a deformation of $E$ extending the automorphism $\omega$, and we are going to denote the deformation functor of such pairs by $(E, \omega)$.
Likewise, $\mdl F$ is a deformation of $F$ with a torsion point lifting the point of $F$ which appears in the definition of the action of $\Gamma$, and we denote the deformation functor of such pairs by $(F, c)$.

The functor $\Jac_{X/C}$ is isomorphic to the product of the deformation functors
$$ (E, \omega) \times (F, c). $$
To write down a versal family for $\Jac_{X/C}$, we treat the problem separately for both factors.

\subsubsection{Deforming elliptic curves with automorphisms}

Let $(\mdl E^{univ}, \omega) \to \Spec(R)$ be the universal deformation of $E$ along with its automorphism $\omega$.
This functor is indeed pro-representable: Given a lifting of $\omega$ exists for a given deformation of $E$, then it is
unique.

If $\ord(\omega) = 2$, then  $\omega$ is the involution, which obviously extends to any deformation of $E$.
Hence in that case $R = W[[t]]$.

If $\ord(\omega)  > 2$ then the $j$-invariant of $\mdl E$ is either 0 or 1728.
If the order is prime to $p$, there are no obstruction against lifting $(E, \omega)$, thus $R = W$.

We treat the remaining cases. First, assume $p = 2$ and $d = \ord(\omega) = 4$.
We know from \cite[Lemma 1.1]{L2} that there is no elliptic curve over $W$, with $j$-invariant $1728$ and good reduction.
This means we have to pass to a ramified extension of $W$. We will work over $R = W[i]$ where $i$ is a primitive forth root of unity. 
The following curve $\mdl E_2$ is taken from \cite[\S 2.A]{L2}
$$ y^2 + (-i + 1)xy - iy = x^3 - ix^2 .$$
It has $j = 1728$ and Discriminant $\Delta = 11 - 2i$, and is therefore of good reduction.

For $p = 3$ and $d = 3$, again by \cite{L2}, there will be no elliptic curve over $W$ with $j$-invariant 0 and good reduction.
So let $R = W[\pi]$, where $\pi^2 = 3$. 
Consider the elliptic curve $\mdl E_3$ given by the Weierstra\ss{} equation
$$ y^2 = x^3 + \pi x^2 + x, $$
whose $j$-invariant is 0 and whose discriminant is $\Delta = -16$. In particular, it has good reduction.

In both cases ($p = 2$ or 3), the curve $\mdl E_p$ has an automorphism of order four or three respectively, since on the generic fibre, automorphisms are given by the action of certain roots of unity, and we have chosen the base rings in such a way, that they contain the necessary roots. An automorphism of the generic fibre extends to the entire family, and its order will not change when restricting to the special fibre, as can be seen by considering an \'etale torsion subscheme of sufficiently high order.

We claim that the elliptic curves over the rings constructed above are the universal families for the deformation problem $(E, \omega)$. This follows from the fact that the respective base rings are the smallest possible extensions of $W$ over which the deformation problem can be solved,  and from the fact that an elliptic curve over a strictly henselian ring is determind by its $j$-invariant.

\subsubsection{Deforming elliptic curves with torsion points}
Now we treat the second factor.
If $p$ does not divide $d$, then a $d$-torsion point lifts uniquely to any deformation of $F$. Therefore $(F, c)$ is pro-represented by  $W[[t]]$.

Assume now that $p$ does divide $d$. By the above, we can assume $d = p^n$. Since $F$ is ordinary, we can use the theory of \emph{Serre-Tate local moduli}. For the special case of ordinary elliptic curves, see \cite[8.9]{KM}. By this theory, we can represent the deformation functor of $F$ by a pair
\[
 \mdl F^{univ} \to \Spec(W[[q - 1]])
\]
satisfying the following property:
For a complete local $W[[q - 1]]$-algebra $\Lambda$, consider the pullback
\[\mdl F = \mdl F^{univ} \otimes_{W[[q - 1]]} \Lambda. \]
By $\Z[q, q^{-1}] \to W[[q - 1] \to \Lambda$ we make $\Lambda$ into a $\Z[q, q^{-1}]$-algebra.
There exists a universal group scheme $T$ over $\Z[q, q^{-1}]$ defined in \cite[8.7]{KM} such that
\[ \mdl F[p^\infty] \simeq T[p^\infty] \otimes_{\Z[q, q^{-1}]} \Lambda. \]

The explicit description of $T$ in \cite[8.7]{KM} implies that the sequence
\begin{equation}\label{torsion}
 0 \to \mu_{p^n} \to  T[p^n] \otimes \Lambda \to \Z/p^n\Z \to 0
\end{equation}
is split if and only if the image of $q$ in $\Lambda$ has a $p^n$-th root. However, (\ref{torsion}) is split if and only if $c$ lifts to $\mdl F$.

We conclude that $W[[q - 1]][\sqrt[p^n]{q}]$ is a versal hull of the functor $(F, c)$.

\begin{prop}\label{split-families}
The versal hull $R$ of the functor $\Jac_{X/C}$ is given as follows:
\begin{center}
\begin{tabular}{l | c c c }
        & $p = 2$                                   & $p = 3$                          & $p > 3$ \\ \hline
$d = 2$ & $W[[{t_E}]] \otimes W[[q - 1]][\sqrt[2]{q}]$  & $W[[t_E]] \otimes W[[t_F]]$  & $W[[t_E]] \otimes W[[t_F]]$   \\
$d = 3$ & $W[[t_E]] \otimes W$                   & $W[\pi] \otimes W[[q - 1]][\sqrt[3]{q}] $    & $W \otimes W[[t_F]]$        \\
$d = 4$ & $W[i] \otimes W[[q - 1]][\sqrt[4]{q}]$      & $W \otimes W[[t_F]] $                   & $W \otimes W[[t_F]]$        \\
$d = 6$ & $W \otimes W[[q - 1]][\sqrt[2]{q}]$         & $W[\pi] \otimes W[[q - 1]][\sqrt[3]{q}]$   & $W \otimes W[[t_F]]$
\end{tabular}
\end{center}

\end{prop}
It is easy to read off and interpret the dimension of the tangent space of the deformation functor. For example in the case where $p = 3$ and $d = 3$ we have $\dim(W[\pi] \otimes W[[q - 1]][\sqrt[3]{q}], k[\epsilon]) = 3$. There is one dimension coming from the deformations of $F$, and the rest is due to relations, coming from obstructions. As explained before, we have $h^1(X, \Theta_X) = 4$ in this case and so one dimension is still missing.

To account for this missing dimension, we have to study all deformations of $X$, not just the Jacobian ones.
This is settled by \thmref{def-functor-mapping}. 
Observe that in all the cases, we have
\[
 h^1(X, \Theta_X) - \dim(\Jac_{X/C}(k[\epsilon])) = h^1(C, \Lie(X/C)).
\]

Therefore $\dim(\Fib(k[\epsilon])) = h^1(X, \Theta_X)$, and its makes sense to ask if the absolute deformation functor of $X$ is isomorphic to $\Fib_{X/C}$. In the next section, we will see that this is indeed the case.

\subsection{Classification of deformations}

The most important step to classify deformations of bielliptic surfaces is to show that for a bielliptic surface $X/C$ over $k$ the functors $\Fib_{X/C}$ and $\Def_X$ are isomorphic.

Denote by $J \to C$ the Jacobian of $X \to C$.
If $d$ is not a power of $p$, the claim follows already, since in that case $\Jac_{J/C}$ is unobstructed, and has the right tangent dimension. 
Hence we get a chain of isomorphisms 
\[
 \Jac_{X/C} \simeq \Fib_{X/C} \simeq \Def_X.
\]

In the case where $d$ is a power of $p$, we have to work with the \'etale covering of $X$.
This is more difficult than in the Kodaira dimension one case, because the \'etale cover of the reduction is an elliptic abelian surface and not every deformation of the covering admits a fibration.

To understand the deformation theory of abelian surfaces, we use $p$-divisible groups.
For the reader's convenience, we repeat some basic definitions and facts: Let $p$ be a prime number, and let $S$ be a scheme. A sheaf of groups for the $fppf$-topology is called a \emph{p-divisible group}, if $G$ is $p$-divisible and $p$-primary, 
i.e. $$G = \varinjlim G[p^n]$$  and the groups $G[p^n]$ are finite flat group scheme over $S$ (see \cite{GroBT} where $p$-divisible groups go by the name ``Barsotti-Tate groups'').
The main examples which we are in fact interested in are $p$-divisible groups associated with abelian schemes.
For an abelian $S$-scheme $A$, we set $$ A[p^\infty] = \varinjlim A[p^n].$$

The deformation theory of abelian schemes is controlled by $p$-divisible groups.
To be precise, let $R$ be a ring in which $p^N = 0$. For a nilpotent ideal $I \subset R$ we define the category $\cat T$ of triples:
\[
 (A, \mdl G, \epsilon)
\]
where $A$ is an abelian scheme over $R/I$, $\mdl G$ is a $p$-divisible group over $R$ and $\epsilon$ an isomorphisms $\mdl G \otimes_R R/I \simeq A[p^\infty]$.
Now we have the theorem of Serre and Tate:

\begin{thm}[Theorem 1.2.1 \cite{KatzST}]\label{Serre-Tate}
 There is an equivalence between $\cat T$ and the category of abelian schemes over $R$ given by
$$ \mdl A \mapsto (\mdl A \otimes_R R/I, \mdl A[p^\infty], \text{ natural } \epsilon) .$$
\end{thm}

We will use the following statement to understand the lifting behavior of morphisms of the latter:

\begin{lemma}[Lemma 1.1.3 \cite{KatzST}]\label{pdiv-endos}
 Let $\mdl G$ and $\mdl H$ be p-divisible groups over $R$. Assume $I^{\nu + 1} = 0$. Let $G$ and $H$ denote their restrictions to $\Spec(R/I)$.  Then the following holds:
\begin{enumerate}
 \item The groups $\Hom_R(\mdl G, \mdl H)$ and $\Hom_{R/I}(G, H)$ have no $p$-torsion.
 \item The reduction map $\Hom_R(\mdl G,\mdl H) \to \Hom_{R/I}(G, H)$ is injective.
 \item For any homomorphism $f \colon G \to H$ there exists a unique homomorphism $\phi_\nu$ lifting $[p^\nu] \circ f$.
 \item In order for $f$ to lift to a homomorphism $f \colon \mdl G \to \mdl H$, it is necessary and sufficient for the homomorphism $\phi_\nu$  to annihilate $\mdl G[p^\nu]$.
\end{enumerate}
\end{lemma}

In the course of the proof, we will use the following lemma:

\begin{lemma}\label{trace-map}
Let $\mdl E$ be an elliptic scheme over a base scheme $\mdl S$. Let $\Omega$ be an automorphism of $\mdl E$ of order $d$. We consider the trace map:
\[
 \tr_\Omega = \Id + \Omega + \dots + \Omega^{d-1}.
\]
Then $\tr_\Omega$ gives the zero map on $\mdl E$.
\end{lemma}
\begin{proof}
 We can prove the statement fibrewise. So assume that $S$ it the spectrum of a field.
It holds \[ (\Id - \Omega) \circ \tr_\Omega = \Id - \Omega^d = 0 \]
However, $\Id - \Omega$ is surjective, hence $\tr_\Omega = 0$ follows.
%
%
\end{proof}

\begin{prop}\label{bielliptic-obstructions}
Let $f \colon X \to C$ be a Jacobian bielliptic surface over $k$.
Then $f$ extends to any deformation $\mdl{X}$ of $X$ over $\Lambda \in \Alg_W$.
\end{prop}
\begin{proof}
We have an \'etale Galois cover $A = E \times F$ of $X$. The Galois group is isomorphic to $\Gamma = \Z/d\Z$.
The action on $E$ is by an automorphism $\omega$, fixing the zero section.
For a deformation $\mdl X$ of $X$, we get a diagram
\begin{equation*}
 \xymatrix{ A \ar[d] \ar[r]  &  \mdl A \ar[d] \\
            X \ar[r]         & \mdl{X} }
\end{equation*}
where the right hand column is the unique lifting of the left. By \cite[Theorem 6.14]{GIT} we can give $\mdl A$ a structure of an abelian scheme, extending the group structure on $A$. 
Now the strategy is as follows: First, we show that $\mdl A$ has an automorphism $\Omega$ lifting $\Id \times \omega$.
 Then we study the action of $\Omega$ on the $p$-divisible group $\mdl A[p^\infty]$, and use the trace map defined by $\Omega$ to lift the projection $E[p^\infty] \times F[p^\infty] \to F[p^\infty]$. This lifting will descend to the desired lifting of the fibration $f$ on $X$.

Denote the image of the generator $1 \in \Gamma$ in $\Aut(\mdl A)$ by $\sigma$. Note that $\sigma$ does not necessarily fix the zero section.
We study its action on $\mdl A$:
Set $c = \sigma(0) \in \mdl A(\Spec(\Lambda))$ and denote by $t_{-c}$ the morphism given by translation the $-c$. We set
\begin{equation*}
  \Omega = \sigma \circ t_{-c} \;\text{ and }\; \Omega' = t_{-c} \circ \sigma .
\end{equation*}
 Both maps fix the zero section of $\mdl A$ and are therefore group automorphisms of $\mdl A$. Furthermore, they lift the automorphism $\Id \times \omega$ of $E \times F$, which implies $\Omega = \Omega'$ since the lift of an automorphism is unique.

This means that $\Omega$ and $t_c$ commute, and since $\sigma$ and $\Omega$ are of order $d$, we get that $c$ is a torsion point of order $d$, which lifts the action of $\Gamma$ by translation on $F$.

To proceed with the proof, we pass to the category of $p$-divisible groups, as explained in \thmref{Serre-Tate}.
Our aim is to lift the second projection $$\pr_2 \colon E[p^\infty] \times F[p^\infty] \to F[p^\infty].$$
We know there exists some integer $N$ such that there exists a unique lift $\phi_N$ of  $[N] \circ \pr_2$ (\lemref{pdiv-endos}).
We compare $\phi_N$ with the trace  $\tr_\Omega$  defined by $\Omega$.
The restriction $\overline{\tr_\Omega}$ of $\tr_\Omega$ to $A[p^\infty]$ gives the map 
$$[d] \circ \pr_2 \colon A[p^\infty] \to F[p^\infty] = \Im(\overline{\tr_\Omega})$$
because $\tr_\Omega$ is  multiplication by $d$ on the factor $F[p^\infty]$ and the zero map on the factor $E[p^\infty]$ (see \lemref{trace-map}).
Now, we get that $[d] \circ \phi_N$ is a lift of $[N] \circ \overline{\tr_\Omega}$.
Again, because an endomorphisms has at most one lifting, it follows
$$[d] \circ \phi_N = [N] \circ \tr_\Omega.$$
Factoring out by $\mdl A[N]$, we see that $\tr_\Omega$ is a lift of $[d] \circ [\pr_2]$.
It remains to show that $\mdl A[d]$ lies in the kernel of $\tr_\Omega$.
To see this, we consider the exact sequence of finite flat group schemes
$$ 0 \to \mdl A[d]^0 \to \mdl A[d] \to \mdl A[d]^{et} \to 0. $$
We first show $\tr_\Omega(\mdl A[d]^0) = 0 $. Again, we have an exact sequence
$$ 0 \to  \mdl A[d]^{mult} \to \mdl A[d]^0 \to \mdl A[d]^{bi} \to 0. $$
The outer groups denote the multiplicative part and the biinfinitesimal part respectively.
The category of multiplicative groups schemes is dual to the category of \'etale group schemes via Cartier duality - thus endomorphisms lift uniquely, and we get $\tr_\Omega(\mdl A[d]^{mult}) = 0$.
Now, we consider the sequence of $p$-divisible groups
 $$ 0 \to \mdl A[p^\infty]^{mult} \to \mdl A[p^\infty]^0 \to \mdl A[p^\infty]^{bi} \to 0. $$
Since $\Omega$ maps $\mdl A[p^\infty]^{mult}$ into itself, we get an induced action of $\Omega$ on $\mdl A[p^\infty]^{bi}$, and in particular, $\tr_\Omega$ descends to $\mdl A[p^\infty]^{bi}$. If this group is non-trivial, it is a lift of $E[p^\infty]$ on which $\tr_\Omega$ is zero. Again by uniqueness of lifts, we get $\tr_\Omega(\mdl A[d]^{bi}) = 0$.

We saw that $\tr_\Omega(\mdl A[d]^0) = 0$ and it remains to show $\tr_\Omega(\mdl A[k]^{et}) = 0$. However, this is clear, since we deal with \'etale group schemes.
We conclude that $\pr_2$ extends to $\mdl X$.
\end{proof}

So far, we have treated only Jacobian bielliptic surfaces. But the non-Jacobian cases are mostly trivial.
Consulting the table of bielliptic surfaces in \cite{BM2}, we see that the  Tate-{\v{S}}afarevi{\v{c}} group is trivial if the associated Jacobian has obstructed deformations, except in one case in charactertic two.

To construct this surface, let $E$ and $F$ be ordinary elliptic curves over $k$ with $p = 2$. We set $A = (E \times F)/ \mu_2$, where $\mu_2$ is the subgroup scheme embedded diagonally into 
$$(E \times F)[2]^0  \simeq \mu_2 \times \mu_2. $$
The quotient $A$ is an abelian surface which does not split into a product.

Let $c$ be a non trivial 2-torsion point of $F$.
The action on $E \times F$, given by $(x, y) \mapsto (x + c, -y)$, commutes with the diagonal action of $\mu_2$ and thereby descends to a $\Z/2\Z$-action on $A$. The bielliptic surface $X$ we are interested in is now given by $A /(\Z/2\Z)$.
The Jacobian of $X$ is clearly $J = (E \times F)/(\Z/2\Z)$.

Now let $\mdl{X}$ be a deformation of $X$. Once more we have a diagram:
\begin{equation*}
 \xymatrix{ A \ar[d] \ar[r]  & \mdl A \ar[d] \\
            X \ar[r]         & \mdl{X} }
\end{equation*}

We claim that $\mdl A$ admits a lifting of the elliptic fibration $f \colon A \to F/\mu_2$:
We have an exact sequence
\begin{equation}\label{ordinary-decomposition}
 0 \to \mdl A[p^\infty]^{tor} \to \mdl A[p^\infty] \to \mdl A[p^\infty]^{et} \to 0 .
\end{equation}
The morphism $A[p^\infty]^{et} \to F[p^\infty]^{et}$ induced by $f$ lifts uniquely to $$\varphi \colon \mdl A[p^\infty]^{et} \to \mdl F[p^\infty]^{et},$$
since we are dealing with \'etale group schemes. Denote by $\mdl B$ the p-divisible group obtained by pushout of (\ref{ordinary-decomposition}) along $\varphi$.
We still have $\mdl A[p^\infty]^{tor} \subset \mdl B$ and inside $\mdl A[p^\infty]^{tor}$ is contained the kernel of the unique lift of $A^{tor} \to F[p^\infty]^{tor}$.
Dividing out $\mdl B$ by that kernel we obtain a lifting of $f$.

As in the proof of \thmref{bielliptic-obstructions}, we see that $f$ descends to $\mdl{X}$. Therefore $\mdl{X}$ is elliptic. To sum up, we have the following theorem:

\begin{thm}\label{thm-bielliptic-fib-lift}
 Every deformation $X$ of a bielliptic surface $X$ induces a lifting of the elliptic fibration $X \to C = \Alb(X)$.
\end{thm}

Next, we show that a versal deformation of a bielliptic surface is algebraizable.

\begin{prop}\label{algebraization}
Let $X$ be a bielliptic surface over $k$. Denote by $\mdl X^{vers} \to \Spf(R)$ a formal versal family of $\Def_X$.
Then there exists a projective scheme $\overline{\mdl X}$ over $R$ such that $\mdl X^{vers}$ is the completion of $\overline{\mdl X}$ at the special fibre.
\end{prop}
\begin{proof}
For an arbitrary deformation $\mdl X$ of $X$, denote by $\mdl A \to \mdl X$ the unique lifting of the abelian covering $A \to X$. 
In the proof of \proref{bielliptic-obstructions}
we saw that the abelian scheme $\mdl A$ has an automorphism $\Omega$ lifting the automorphism $\omega \times \Id$ of $E \times F$.

The automorphism of $\mdl X$, obtained from $\Omega$ by descent, will again be denoted by $\Omega$.
Now $\Omega$ is a $\mdl C$-automorphism of $\mdl X$, i.e, its action is confined to the fibres of the fibration.

We claim that the fixed locus of $\Omega$ is flat over $\mdl C$:
Every closed point $x \in \mdl C$ has an \'etale neighborhood $\mdl U \to \mdl C$ such that the pullback 
\[ \mdl X_{\mdl U} = \mdl X \times_{\mdl C} \mdl U \]
can be given the structure of an abelian scheme, in such a way that the base change of $\Omega$ to $\mdl X_{\mdl U}$ becomes a group automorphism.
We consider the endomorphism $\Omega - \Id$ of $\mdl X_{\mdl U}$. It is a surjective map of abelian schemes and therefore flat by (\cite[Lemma 6.12]{GIT}). In particular its kernel, i.e the fix locus of $\Omega$, is flat over $\mdl U$.

Thus we have found a relative Cartier divisor of $\mdl X \to \mdl C$.  Its degree can be computed on the reduction.
It equals the order of the subgroup scheme of $E$ fixed by $\omega$. In particular it is positive, which means that $\mdl Z$ is a relatively ample divisor for $\mdl X \to \mdl C$.

Now denote by $\mdl X^{vers} \to \Spf(R)$ a versal family of $\Fib_{X/C}$. It is a formal scheme over the hull of the deformation functor $\Fib_{X/C}$, and we conclude form \thmref{thm-bielliptic-fib-lift} that it admits an elliptic fibration $F \colon \mdl X^{vers} \to \mdl C$ lifting $X \to C$.

Denote by $\ideal m$ the maximal ideal of $R$ , and set $X_n = \mdl X^{vers} \otimes_R R/\ideal m^{n+1}$.
The construction of $\mdl Z$ gives rise to a compatible system of relatively ample invertible sheaves $\O_{X_n}(Z_n)$.
Tensoring with the invertible sheaf coming from the divisor of a fibre of $\mdl X^{vers} \to \mdl C$, we obtain a system of ample line bundles $\sheaf H_n$.
Thus by Grothendieck's Algebraization Theorem \cite[Theorem 4.10]{Ill1}, we conclude that $\mdl X^{vers}$ is the completion of some projective scheme $\overline { \mdl X^{vers}}$ over $\Spec(R)$.
\end{proof}

The above proposition helps us to answer another natural question: $X$ is called bielliptic because it has two transversal elliptic fibrations.
Namely the smooth one, denoted by $f$, coming from the projection $E \times F \to F$ and a second one, denoted by $g$, with base curve $\P^1_k$, coming from $E \times F \to E$.  We saw that the first fibration is preserved under deformation, but what about the second one?

\begin{prop}\label{bielliptic-deformations}
 Let $X$ be a bielliptic fibration. Then every deformation $\mdl X$ of $X$ extends both elliptic fibration.
\end{prop}
\begin{proof}
We are going to show that the versal deformation $\mdl X^{vers} \to \Spf(R)$ admits an extension of $g$. Then the claim follows by versality.

Denote by $K$ the fraction field of $R$.
We can use surface theory to analyze the generic fibre $\overline{\mdl X}_K$ of the algebraization $\overline{\mdl X}$ of $\mdl X^{vers}$. 
Denote by $\sheaf L = \O_{\overline{\mdl X}}(\mdl Z)$ the line bundle associated to the divisor $\mdl Z$, constructed in the proof of \proref{algebraization}.
The canonical bundle of $\overline{\mdl X}_K$ has self-intersection number 0.
It follows that the line bundle $\sheaf L_K^{\otimes m}$, gives rise to an elliptic fibration $g' \colon \overline{\mdl X}_K \to \P^1_K$, if we choose $m$ sufficiently big \cite[Theorem 7.11]{Bad}.

Since $\overline{\mdl X}$ is proper and normal, we can extend $g'$ to a rational map $g' \colon \overline{\mdl X} \to \P^1_R$ which is defined on a non-empty open subset intersecting the special fibre.
Now, there are sections $s_1, s_2 \colon \Spec(R) \to \P^1_R$ whose associated closed subschemes are disjoint and who do lie inside the image of $g'$.
Taking the closures of the inverse images of those sections under $g'$, we get two divisors $\mdl G_1$ and $\mdl G_2$ in $\overline{\mdl X}$ who have disjoint specializations on a non empty open subset of the special fibre (namely the locus where $g'$ is defined).

We claim that their reductions $G_1$ and $G_2$ are irreducible (and hence disjoint):
The group of divisors of $X$ modulo numerical equivalence is generated by two classes $F$ and $G$, where $F$ is a fibre class of $f$ and $G$ is one of $g$. The intersection numbers are
\begin{align*}
 F \cdot F = 0, && F \cdot G > 0, && G \cdot G = 0.
\end{align*}
In particular, there are no effective divisors on $X$ with negative self-intersection.
It follows that the specialization of a curve of canonical type is again of canonical type. However, every curve of canonical type on $X$ is irreducible, hence the claim follows.

Considering the global sections of $\sheaf L^{\otimes m}$ associated to the effective divisors $\mdl G_1$ and $\mdl G_2$, we find that $\sheaf L^{\otimes m}$ is globally generated. It follows that the map given by $\sheaf L^{\otimes m}$ is in fact a morphism, lifting $g$.
\end{proof}

We illustrate the theorem by looking at a special case:
Let $k$ be of characteristic three, and let $X$ denote the Jacobian bielliptic surface of index $d = 3$ over $k$.
What does the fibre $\mdl X_\eta$ of the versal family of $X$ over the generic point of the base look like?
The smooth fibration with elliptic base curve does not have a section. The three sections which appear when we do base change with the algebraic closure of $\eta$ do not descend to $\mdl X_\eta$. Instead, we have a multi-section of degree three.

There is an explicit construction of a bielliptic surface with $d = 3$ over $\Q$, which shows the same behaviour. It was given in \cite{MR2053453} as a counterexamples to the Hasse principle which cannot be explained by the Manin obstruction.

\bibliography{RefSmooth}

\end{document}